\documentclass[11pt,a4paper,reqno,english]{ijocta}

\usepackage[dvips]{graphicx}
\usepackage[margin=2cm]{geometry}
\usepackage[font=small,labelfont=bf,tableposition=top]{caption}
\usepackage[font=footnotesize]{subcaption}
\usepackage{multicol}
\usepackage{babel}
\usepackage{amssymb,amsmath,amsfonts}
\usepackage{multirow}
\usepackage{blindtext} 
\usepackage{adjustbox}
\usepackage{hyperref}
\usepackage{url}
\usepackage{algorithmic}

\usepackage{tikz}

\newtheorem{theorem}{Theorem}

\newtheorem{algorithm}{Algorithm}

\newtheorem{corollary}{Corollary}

\newtheorem{definition}{Definition}

\newtheorem{lemma}{Lemma}

\newtheorem{proposition}{Proposition}

  \def \cl{{\textup{cl}}}
\def \int{ {\textup{int}}}
\def \ri{ {\textup{ri}}}
  \def \argmin{{\textup{argmin}}}

\usetikzlibrary{decorations.pathreplacing}

\tikzset{
pics/realline/.style 2 args = {
    code = {\draw [thick]   (0,0) -- (6,0) node [right=2mm] {$\mathbb{R}$};
            \fill[black]    (1,0) circle (1mm) node[above=2mm] {$#1$}
                            (3,0) circle (1mm) node[above=2mm] {$#2$}
                            (6,0) circle (1mm) node[above=2mm] {$V_4$};
            \foreach \i [count=\j] in {0,0.5,1,1.5,3,4.5,6} 
                        \coordinate (-\j) at (\i,0);
    }}}
    
\begin{document}


\vspace{1.42cm}

\title[\small{\textit{Conic Reformulations for KL Divergence Constrained  DRO and Applications}}]
{\textbf{Conic Reformulations for Kullback-Leibler Divergence Constrained  Distributionally Robust Optimization  and Applications}}
\author[\small{\textit{B. Kocuk\\}}]
{\normalsize Burak Kocuk \vspace{0.58cm} 
\\ \small
Industrial Engineering Program, Faculty of Engineering and Natural Sciences, Sabanc{\i} University, \\ Istanbul, Turkey \\
Email: burak.kocuk@sabanciuniv.edu \vspace{0.1cm} \\
}
  
\thanks{Corresponding Author. Email: burak.kocuk@sabanciuniv.edu}

\begin{abstract}
In this paper, we consider a distributionally robust optimization (DRO) model in which the ambiguity set is defined as the set of distributions whose Kullback-Leibler (KL) divergence to an empirical distribution is bounded. Utilizing the fact that KL divergence is an exponential cone representable function, we obtain the robust counterpart of the KL divergence constrained  DRO problem as a dual exponential cone constrained program under mild assumptions on the underlying optimization problem. The resulting conic reformulation of the original optimization problem can be directly solved by a commercial conic programming solver. We specialize our generic formulation to two classical optimization problems, namely, the Newsvendor Problem and the Uncapacitated Facility Location Problem. Our computational study in an out-of-sample analysis shows that the solutions obtained via the DRO approach yield significantly better performance in terms of the dispersion of the cost realizations while the central tendency deteriorates only slightly compared to the solutions obtained by stochastic programming.

\vspace{7pt}
\noindent
\textbf{Keywords: }{Distributionally robust optimization, Stochastic programming, Conic programming, Newsvendor Problem, Uncapacitated Facility Location Problem}

\vspace{2pt}
\end{abstract}

\maketitle

\begin{multicols*}{2}

\section{Introduction}

Decision making under uncertainty is one of the most challenging tasks in operations research. 
Two paradigms are predominantly used in the  literature to address  uncertainty: stochastic programming and robust optimization. In the classical stochastic programming \cite{BL2011, SDR2014}, a predefined set of scenarios (or samples) are determined, either taken directly from observed data or after fitting an appropriate distribution. Then, the objective function is replaced with an expectation taken with respect to the random elements, and constraints are copied for each scenario.  In addition to the  assumption about knowing the underlying distribution, this basic stochastic programming approach has some limitations: Firstly, the size of the resulting deterministic equivalent grows larger with the size of the scenarios. Secondly, an expectation may not be an appropriate performance measure. Thirdly, satisfying constraints for each scenario might be too restrictive. The respective remedies for these shortcomings are proposed such as sample average approximation to limit the problem size, risk-averse objective function for a more appropriate performance measure and chance constraints to allow constraint satisfaction with high probability. However, the implicit assumption of stochastic programming remains, which is the need to \textit{assume} a distribution by analyzing the data or \textit{fitting} one. Unfortunately, this step may not be performed satisfactorily in all cases.

In robust optimization \cite{BenTal2002, BEN2009, Bertsimas2011}, a predefined uncertainty set, which includes all possible values of the uncertain elements, is used. Then, the optimization is performed with the aim of optimizing with respect to  the worst possible realization from the uncertainty set. There are two main advantages of using robust optimization. Firstly, the decision maker does not need to make any assumptions about the distribution of the uncertain elements in the problem as opposed to the stochastic programming approach. Secondly,  the deterministic equivalent (or so-called the \textit{robust counterpart}) of the robust optimization problem typically has the same computational complexity as the deterministic version of the problem under reasonable assumptions on the uncertainty sets. On the other hand, the main disadvantage of the robust optimization approach is that depending on the construction of the uncertainty set, it might lead to overly conservative solutions, which might have poor performance in central tendency such as expectation.

Distributionally robust optimization (DRO) is a relatively new paradigm that aims to combine stochastic programming and robust optimization approaches. The main modeling assumption of DRO is that some \textit{partial}  information about the  distribution governing the underlying uncertainty is available, and the optimization is performed with respect to the worst distribution from an \textit{ambiguity set}, which contains all distributions consistent with this partial information. There are mainly two streams in the DRO literature based on how the ambiguity set is defined: moment-based and distance-based.

In moment-based DRO, ambiguity sets are defined as the set of  distributions whose first few moments are assumed to be known or constrained to lie in certain subsets. If certain structural properties hold for the ambiguity sets such as convexity (or conic representability), then tractable convex (or conic reformulations) can be obtained  \cite{Popescu2007, Delage2010, Wiesemann2014}. Recently,   chance constrained DRO problems also draw attention \cite{Zymler2013, Xie2018}. 
In distance-based DRO, ambiguity sets are defined as the set of distributions whose distance (or divergence) from a reference distribution is constrained. For  Wasserstein distance \cite{Gao2016, Mohajerin2018, Hanasusanto2018, Xie2019}  and $\phi-$divergence \cite{BenTal2013, Jiang2016, Lam2019} constrained DRO, tractable convex reformulations have been proposed. 

As summarized above, in many cases, tractable convex robust counterparts or reformulations can be obtained for robust and DR optimization problems. However, an even more special structure such as conic representability can be preferred whenever available. Especially, if the robust counterpart can be expressed as a conic program for which the underlying cone admits a self-concordant barrier, then efficient polynomial-time interior point methods can be applied directly \cite{NM1994}. This desired property holds for linear programs, second-order cone programs and semidefinite programs, which appear extensively in both robust and DR optimization literature. We note that the efficiency of the conic programming solvers specialized in these three problem classes has improved considerably.

There is  some recent interest in  conic programs for which the underlying cone is not self-dual, such as exponential cone. There are two main reasons: i) exponential cone has extensive expressive power that is useful to model optimization problems involving the exponential and logarithm functions (see e.g. \cite{Serrano2015}), and ii) a practical implementation of a primal-dual interior point method is developed \cite{Dahl2019} although its polynomial-time complexity has not proven yet. 
Our paper will exploit both the expressive power of the exponential cone and the practical implementation that can be used to solve the resulting optimization problem, as detailed below.

The Kullback-Leibler (KL) divergence \cite{KL1951} is a popular divergence measure in information theory that can be used to  quantify the divergence of one distribution from another (see Definition~\ref{def:KL divergence}) and we prove that it is exponential cone representable (see Definitions~\ref{def:exp cone} and \ref{def:conic repr func},  and Proposition~\ref{prop:D_KLisExpConeRepr}). 
 Although the robust counterpart of  KL divergence constrained DRO   is proven to be a tractable convex program \cite{Hu2013}, to the best of our knowledge, its exponential cone representability has not been exploited in the literature before. Also, its practical performance against stochastic programming have not been analyzed in detail except for a limited number of applications from power systems  \cite{Chen2018, Li2018}.

In this paper, we consider KL divergence constrained  DRO problems and propose their dual exponential cone constrained reformulation under the mild assumption of conic representability. This  allows us to solve the corresponding robust counterpart using a conic programming solver such as MOSEK \cite{MOSEK2020py}. We also present how the generic formulation can be specialized for two classical problems: Newsvendor  and Uncapacitated Facility Location. Although the DRO methodology has been applied to variations of these problems \cite{Hanasusanto2015, Natarajan2018, Lee2020, Lu2015, Santi2018, Basciftci2019}, to the best of our knowledge, their KL divergence constrained versions have not been studied in detail.
 Our computational results suggest that solutions obtained via a DR approach give slightly higher cost realizations when central tendencies such as mean and median are considered compared to  solutions obtained via stochastic programming in an out-of-sample analysis. However, the dispersion (measured by the standard deviation and range of the cost realizations) and the risk (measured by the average of worst cost realizations and the third quartile) metrics improve significantly with solutions obtained via a  DR approach.

The rest of the paper is organized as follows: In Section~\ref{sec:prelim}, we review basic concepts from convex analysis and probability theory which serve as the basis of our main result about conic reformulation of KL divergence constrained DRO problems  in Section~\ref{sec:main}. Then, we analyze two applications, namely, the Newsvendor Problem in Section~\ref{sec:newsvendor} and the Uncapacitated Facility Location Problem in Section~\ref{sec:ufl}, and present the results of our computational study. Finally, we conclude our paper in Section~\ref{sec:conc}.

 

\section{Preliminaries}
\label{sec:prelim}

Before stating our main result in Section~\ref{sec:main}, we will first review some important concepts from convex analysis in Section~\ref{sec:prelim convex} and probability theory in Section~\ref{sec:prelim prob}.

\subsection{Convex Analysis}
\label{sec:prelim convex}
For a set $X\subseteq\mathbb{R}^m$, we denote its interior as $\int(X)$, its relative interior as $\ri(X)$ and 
its closure as $\cl(S)$. We use the shorthand notation $[n]$ for the set $\{1,\dots,n\}$.

We will first review  some basic concepts from convex analysis related to cones.
\begin{definition}[Regular cone]
	A cone $K \subseteq  \mathbb{R}^m$ is called regular if it is closed, convex, pointed and full-dimensional.
\end{definition} 
Examples of regular cones include the nonnegative orthant, Lorentz (or second-order) cone and the cone of positive semidefinite matrices. We will refer to these cones as \textit{canonical cones} in this paper.

\begin{definition}[Dual cone]
	The dual cone to a cone $K \subseteq \mathbb{R}^m$ is defined as $K_* = \{y\in \mathbb{R}^m: x^Ty \ge 0, \ \forall x \in K\}$.
\end{definition}
It is well-known that the dual cone to a regular cone is also regular. In addition, the three canonical cones mentioned above are self-dual.

We will now define the exponential cone, which is the key ingredient of this paper.
\begin{definition}[Exponential cone]\label{def:exp cone}
	The exponential cone, denoted as $K_{\text{exp}}$, is defined as 
\begin{equation*}
\begin{split}
K_{\exp{}} &= \cl \big(\{ x \in \mathbb{R}^3: x_1 \ge x_2 e^{x_3/x_2}, \ x_2 > 0 \} \big).
\end{split}
\end{equation*}
\end{definition}

As opposed to the three canonical cones mentioned above, the exponential cone is not self-dual although it is a regular cone.

\begin{proposition}[See e.g. \cite{Serrano2015}]
The dual cone to the exponential cone (or simply the dual exponential cone) is given as
\begin{equation*}
\begin{split}
(K_{\exp{}})_* &= \\
 \cl \big( &\{ s \in \mathbb{R}^3: s_1 \ge -s_3 e^{(s_2-s_3)/s_3}, \ s_3 < 0 \} \big) .
\end{split}
\end{equation*}
\end{proposition}

The following definitions are  instrumental in the description of conic programming problems:
\begin{definition}[Conic inequality]
	A conic inequality with respect to a regular cone $K$ is defined as $x \succeq_{K} y$, meaning that  $x - y \in K$. 
	We will denote the relation $x \in \int(K)$ alternatively as  $x \succ_K 0$.
\end{definition}
\begin{definition}[Conic representability of a set]
A set $X \subseteq \mathbb{R}^n$ is called conic representable if it can be expressed as
	\[
	X = \{ x\in\mathbb{R}^n : \exists y\in\mathbb{R}^k : Ax + B y \succeq_K b\},
	\]
	for some appropriately chosen regular cone $K$.

If any of the variables in this representation is integer, then $X$ is called a mixed-integer conic representable set.
\end{definition}

\begin{definition}[Conic representability of a function]\label{def:conic repr func}
A function is called conic representable if its epigraph is a conic representable set.
\end{definition}

In this paper, when we say that ``a set or function is conic representable'', we will implicitly assume that the cone used in its representation is either one of the three canonical cones or the (dual) exponential cone.


\subsection{Probability Theory}
\label{sec:prelim prob}

\begin{definition}[Probability simplex]
The probability simplex in dimension $S$ is denoted as 
\[
\Delta^S := \big\{ p\in\mathbb{R}^S_+ : \ \sum_{s=1}^S p_s = 1 \big \}.
\]
\end{definition}


The following function can be used to measure the ``divergence'' of one distribution from another.
\begin{definition}[KL Divergence]\label{def:KL divergence}
For two discrete probability distributions $p\in \Delta^S$ and $q \in \ri(\Delta^S$), the KL divergence from $p$ to $q$ is defined as
\[
D_{KL} ( p || q) := \sum_{s=1}^S p_s \log (p_s / q_s ).
\]
\end{definition}
We note that the KL Divergence does not define a distance metric between two probability distributions since it is not symmetric. However, it has the following useful property.
\begin{proposition}\label{prop:D_KLisExpConeRepr}
Let $p\in \Delta^S$ and $q \in \ri(\Delta^S$). Then, the function  $D_{KL} ( p || q) $ is exponential cone representable.
\end{proposition}
\begin{proof}
Due to Definition~\ref{def:conic repr func}, it suffices to show that the epigraph of the function  $D_{KL} ( p || q) $ is an exponential cone representable set.
Since the set $\{(x,y,t)\in\mathbb{R}^3: t\ge x\log(x/y)\}$ has the exponential cone representation $(y,x,-t)\in K_{\text{exp}}$ \cite{MOSEK2020}, we obtain an exponential cone representation for  the function  $D_{KL} ( p || q) $ as follows:
\begin{equation*}\label{eq:D_KLisExpConeRepr}
\begin{split}
& \{ (p,q,\epsilon) \in  \Delta^S \times  \ri(\Delta^S) \times \mathbb{R} :  D_{KL} ( p || q)  \le \epsilon  \} \\
= & \big\{ (p,q,\epsilon) : \exists \delta \in \mathbb{R}^S : \sum_{s=1}^S \delta_s \le \epsilon, \\
& \hspace{1.5cm} (q_s, p_s, -\delta_s) \in K_{\text{exp}}, \  s\in [S] \big\}.
\end{split}
\end{equation*}
\end{proof}

The following proposition gives an upper bound on the KL-divergence of a given distribution from any other distribution.
\begin{proposition}\label{prop:D_KL_max}
Let  $q \in \ri(\Delta^S$). Then, 
\[
\overline{\epsilon}(q) := \sup_{ p \in \Delta^S} \{ D_{KL} ( p || q)  \} =   \log(1 / \min_{s\in[S]}\{q_s\}) .
\]
\end{proposition}
\begin{proof}
Notice that the objective function of the optimization problem $\sup_{ p \in \Delta^S} \{ D_{KL} ( p || q) \}$ is convex and its feasible region is a polytope. Therefore, there exists an optimal solution which is an extreme point of $\Delta^S$. Note that the extreme points of  $\Delta^S$ are the unit vectors of $\mathbb{R}^S$, denoted by $e_s$ for $s\in[S]$. Then, we have 
\begin{equation*}
\begin{split}
\overline{\epsilon}(q) & =  \max_{s\in[S]}  \{ D_{KL} ( e_s || q) \} \\
& =  \max_{s\in[S]}  \{ \log (1 / q_s ) \} \\
& =  \log(1 / \min_{s\in[S]}\{q_s\}).
\end{split}
\end{equation*}
In the second equality, we use the fact that $\lim_{x \to 0^+} x \log (x/y) = 0$ for $y>0$.
\end{proof}
Proposition~\ref{prop:D_KL_max} is useful to quantify the ambiguity sets in KL divergence constrained DRO problems as we will see later.

\section{Main Results}
\label{sec:main}

In this section, we present our main result about the reformulation of a  KL divergence constrained DRO problem as a conic program under mild conditions. 

\subsection{Generic Problem Formulation}\label{sec:genericProb}
We  first give the generic problem setting  considered in this paper. 
Suppose that there are $m$ random variables $\xi^i \in \mathbb{R}$, $i\in [m]$, each with a discrete distribution $q^i \in \ri(\Delta^{S_i})$ estimated from the historical data as 
\[
\Pr(\xi^i = d_{s}^i) = q^i_s  \quad  s\in [S_i],
\]
where $\{d_s^i : s \in [S_i]\}$ is the set of observed realizations of $\xi^i$, $i\in [m]$. 
%
%
Under this probabilistic setting, we define the ambiguity set
\begin{equation}
\mathcal{P}^i(q^i, \epsilon^i) := \{ p^i \in \Delta^{S_i} : \ D_{KL} ( p^i || q^i) \le \epsilon^i \},
\end{equation}
for  $i\in [m]$, where $\epsilon^i \in \mathbb{R}_+$ controls the divergence from the historical data (or robustness level). Then, we consider  the following KL divergence constrained DRO problem,
\begin{equation}\label{eq:generic}
\min_{y \in \mathcal{Y}} \big\{ h(y) + \sum_{i=1}^m \max_{ p^i \in  \mathcal{P}^i(q^i, \epsilon^i) } \mathbb{E}_{\xi^i} [ H^i (y, \xi^i) ] \big\}, 
\end{equation}
where each expectation is taken with respect to an ambiguous distribution $ p^i \in \mathcal{P}^i(q^i, \epsilon^i)$.
In problem~\eqref{eq:generic}, $h$ is a real-valued function defined on $\mathbb{R}^n$;   $H^i$ is a real-valued function defined on $\mathbb{R}^n \times \mathbb{R}$; and $\mathcal{Y}$ is a subset of $\mathbb{R}^n$. Observe that given $y$ decisions, the inner maximization problem is decomposable over the random elements~$\xi^i$, $i\in [m]$. 

\subsection{Robust Counterpart and Conic Reformulation} 
We will now obtain the robust counterpart \cite{BEN2009} of problem~\eqref{eq:generic} utilizing Conic Duality under mild conditions.
\begin{theorem}\label{thm:main}
Consider the KL divergence constrained DRO problem~\eqref{eq:generic} as described in Section~\ref{sec:genericProb}, and assume that $\epsilon^i > 0$, $i\in[m]$. Then, the robust counterpart is given as follows:
 \begin{subequations} \label{eq:genericReform}
\begin{align}
\min \ &  h(y) + \sum_{i=1}^m \bigg [ \alpha^i + \epsilon^i \beta^i + \sum_{s=1}^{S_i} q_s^{i} u_s^i \bigg ]  \\
\text{s.t.}
\ & \alpha^{i}  - v_s^i \ge  H^{i} (y, d_s^{i})   \hspace{0.5cm}  i \in [m]; s \in [S_i]  \label{eq:genericReform1} \\
\ & \beta^i + w_s^i = 0   \hspace{1.7cm} i\in [m]; s\in [S_i]   \label{eq:genericReform2} \\
\ &  \alpha^i \in \mathbb{R}, \beta^i \in \mathbb{R}_+  \hspace{1.0cm} i\in [m]  \label{eq:genericReform3} \\
\ & (u^i_s, v^i_s, w^i_s) \in   (K_{\exp{}})_*  \hspace{0.1cm} i\in [m]; s\in [S_i]    \label{eq:genericReform4}  \\
\ &   y \in \mathcal{Y}   .
\end{align}
\end{subequations}
\end{theorem}
\begin{proof}
We will start the proof by analyzing the inner maximization problems.
 Given a vector $y\in\mathcal{Y}$,  let us write the $i$-th inner maximization problem explicitly as the following exponential cone constrained program:
\begin{subequations} \label{eq:innerMax_i}
\begin{align}
\max \ & \sum_{s=1}^{S_i} H^{i} (y, d_s^{i})  p_s^{i} \\
\text{s.t.}
\ & \sum_{s=1}^{S_i}  p_s^{i} = 1  \label{eq:innerMaxCons1} \\
\ & \sum_{s=1}^{S_i}  \delta_s^{i} \le \epsilon^i  \label{eq:innerMaxConsEps} \\
\ & \begin{bmatrix} 0 & 0 \\ -1 & 0 \\ 0 & 1 \end{bmatrix} \begin{bmatrix} p_s^i \\ \delta_s^i \end{bmatrix} \preceq_{K_{\text{exp}}}  
\begin{bmatrix} q_s^i \\ 0 \\ 0 \end{bmatrix} \ s\in [S_i] \label{eq:innerMaxConsConic} \\
\ & p_s^i \in \mathbb{R}_+,  \  \delta_s^i \in \mathbb{R} \hspace{1.9cm} s\in [S_i] . \label{eq:innerMaxConsNonneg} 
\end{align}
\end{subequations}
Here, constraints~\eqref{eq:innerMaxCons1}-\eqref{eq:innerMaxConsNonneg}  model the relation $p^i \in \mathcal{P}^i(q^i, \epsilon^i)$, as  stated in Proposition~\ref{prop:D_KLisExpConeRepr}.

Recall that $\epsilon^i > 0$ for each $i\in [m]$. Then, each  inner maximization problem~\eqref{eq:innerMax_i} satisfies essential strict feasibility \cite{BN2001} (e.g. consider $p^i_s = q^i_s$ and $\delta_i^s = \epsilon^i / |S_i|$ for $s\in[S_i]$), and its optimal value is bounded above (e.g. by $\max_{s\in[S_i]} H^i(y, d^i_s)$). Therefore, strong duality holds between problem~\eqref{eq:innerMax_i} and its conic dual given as follows:
\begin{subequations} \label{eq:innerMax_iDual}
\begin{align}
\min \ &  \alpha^i + \epsilon^i \beta^i + \sum_{s=1}^{S_i} q_s^{i} u_s^i \\
\text{s.t.}
\ & \alpha^{i}  - v_s^i \ge  H^{i} (y, d_s^{i})   \hspace{1.0cm}  s\in [S_i]   \\
\ & \beta^i + w_s^i = 0   \hspace{2.2cm}  s\in [S_i]   \\
\ &   \alpha^i \in \mathbb{R}, \beta^i \in \mathbb{R}_+ \\
\ & (u^i_s, v^i_s, w^i_s) \in   (K_{\text{exp}})_*  \hspace{0.6cm} s\in [S_i] .  
\end{align}
\end{subequations}
 Here,  $\alpha^i$, $\beta^i$ and $(u^i_s, v^i_s, w^i_s)$ are the dual variables  associated with the  primal constraints \eqref{eq:innerMaxCons1}, \eqref{eq:innerMaxConsEps} and \eqref{eq:innerMaxConsConic}, respectively. Notice that problem~\eqref{eq:innerMax_iDual} is a dual exponential cone constrained program.
 
 As the final step in the proof, we write the dual of each inner maximization problem and obtain the robust counterpart of problem~\eqref{eq:generic} as problem~\eqref{eq:genericReform}.
\end{proof}
 
We will now discuss the consequences of Theorem~\ref{thm:main} under additional  structural properties such as convexity and conic representability.
 
\begin{corollary} \label{cor:convex}

Consider the KL divergence constrained DRO problem~\eqref{eq:generic} as described in Theorem~\ref{thm:main}. In addition, let us assume that $\mathcal{Y}$ is a convex set, $h(y)$  and $H^i(y,\xi^i)$ are convex functions in $y$, $i\in[m]$. Then, the robust counterpart~\eqref{eq:genericReform} is a convex program.
\end{corollary} 

\begin{corollary} \label{cor:conic}

Consider the KL divergence constrained DRO problem~\eqref{eq:generic} as described in Theorem~\ref{thm:main}. In addition, let us assume that $\mathcal{Y}$ is a (mixed-integer) conic representable set, $h(y)$  and $H^i (y, \xi^i)$ are conic representable functions, $i\in[m]$. Then, the robust counterpart~\eqref{eq:genericReform} is a dual exponential cone constrained (mixed-integer) program.
\end{corollary} 

As an application of Corollary~\ref{cor:conic}, we will consider the Newsvendor Problem in Section~\ref{sec:newsvendor} and the Uncapacitated Facility Location Problem in Section~\ref{sec:ufl}. The common characteristic of these two problems is that the set $\mathcal{Y}$ is a mixed-integer linear set, the function $h$ is a linear function and the functions $H^i$ are the maxima of linear functions (hence, they are polyhedrally representable).
%
%
%

\section{Application to the Newsvendor Problem}\label{sec:newsvendor}

In this section, we analyze 
  a toy example,   the KL divergence constrained DR version of the single-period, single-product Newsvendor Problem. 
In this case, since there is only one random variable $\xi$ (that is, $m=1$), representing the unknown demand, we will omit the superscript $i$ for convenience.

\subsection{Problem Formulation}

Consider the generic formulation~\eqref{eq:generic} with the following specifications:
We let $y \in  \mathcal{Y} := \mathbb{Z}_+$ be the order quantity, and consider  functions 
\[
h(y) := c y,
\]
where $c$ is the variable order cost, and
\[
H(y, \xi) := c_b \max\{\xi - y, 0\} + c_h \max \{y-\xi, 0\},
\]
where $c_b$ is the back-order penalty for unsatisfied demand and $c_h$ is the inventory cost. Notice that  $H(y,\xi)$ is a piecewise linear convex function in $y$ and can be rewritten as
\[
H(y, \xi) = \max\{ -c_b y + c_b \xi, c_h y - c_h \xi \}.
\]
This observation will be useful to linearize constraint~\eqref{eq:genericReform2}.

By omitting $i$ indices and simplifying the notation of problem~\eqref{eq:genericReform} by taking into account the special structure of the newsvendor problem, 
we obtain the following dual exponential cone constrained MIP as its robust counterpart:
\begin{subequations} \label{eq:newsReform}
\begin{align}
\min \ &  cy + \big[ \alpha + \epsilon \beta + \sum_{s=1}^{S} q_s u_s \big] \\
\text{s.t.}
\ & \alpha  - v_s \ge   -c_b y + c_b d_s     \hspace{0.5cm}  s\in [S]   \\
\ & \alpha  - v_s \ge    c_h y - c_h  d_s   \hspace{.75cm}  s\in [S]   \\
\ & \beta + w_s= 0   \hspace{2.2cm}  s\in [S]   \\
\ &  y \in \mathbb{Z}_+, \alpha \in \mathbb{R}, \beta \in \mathbb{R}_+ \\
\ & (u_s, v_s, w_s) \in   (K_{\text{exp}})_*  \hspace{0.55cm} s\in [S] .  
\end{align}
\end{subequations}

\subsection{Computations}

\subsubsection{Experimental Setup}
To compare the effect of robustness level of KL divergence constrained DR version of the Newsvendor Problem, we propose Algorithm~\ref{alg:news}. Note that setting $\epsilon=0$ in problem~\eqref{eq:newsReform} reduces it to stochastic program while
 and larger values of $\epsilon$ indicate more robustness (and conservativeness) in solutions.

\begin{algorithm}
\label{alg:news}
\begin{algorithmic}[1]
\REQUIRE A probability distribution $\mathcal{D}$, the number of samples $R$, the set of robustness levels  $\mathcal{T}$.
\STATE Sample $R$  random variates from $\mathcal{D}$ for training, and obtain the empirical distribution $q$ and the maximum KL divergence $\overline\epsilon(q)$ as computed in Proposition~\ref{prop:D_KL_max}.
\STATE Solve problem~\eqref{eq:newsReform} with  $\epsilon := \theta \overline\epsilon(q)$ for each $\theta \in \mathcal{T}$ to obtain a decision $y^*(\theta)$. 
\STATE Sample $R$  random variates from $\mathcal{D}$ for testing, and then compute the cost realizations for each realization under the decision $y^*(\theta)$. 
\end{algorithmic}
\end{algorithm}

We implement Algorithm~\ref{alg:news} in the Python programming language and use MOSEK 9.2 \cite{MOSEK2020py} to solve the dual exponential cone constrained MIP problem~\eqref{eq:newsReform}.

\subsubsection{Results}
\label{sec:newsResults}

For this illustration, we choose the following cost coefficients:
\[
c=1, \ c_b=2, \ c_h=1.
\]
We will now specify the parameters of Algorithm~\ref{alg:news}.
Firstly, we experiment with three different discrete distributions:
\begin{itemize}
\item Discrete Uniform Distribution with parameters 0 and 10, denoted as \texttt{Uniform(0, 10)}.
\item Binomial Distribution with parameters 10 and 0.5, denoted as {\texttt{Binomial(10, 0.5)}}.
\item Poisson Distribution with parameter 5, denoted as {\texttt{Poisson(5)}}.
\end{itemize}
We sample $R=100$ random variates separately to obtain ``training'' and ``test'' datasets. 
Then, we repeat the experiments for the following ``robustness'' levels:
\[
\mathcal{T} := \{0.00, 0.05, 0.10, 0.15, 0.20, 0.25\}.
\]

The summary statistics of our experiments are reported in Tables~\ref{tab:newsUnif}-\ref{tab:newsPoi}  for  Uniform, Binomial and Poisson distributions, respectively. In particular, we report the average and standard deviation of the cost realizations, abbreviated as ``Avg.'' and ``St. Dev.'', respectively. In addition, we compute the average of the worst 10\% of the realizations, abbreviated as ``Worst 10\%'', to quantify  the risk.
We observe that as the robustness level~$\theta$ increases, the optimal order quantity $y^*$ increases (recall that $\theta=0.00$ corresponds to the stochastic programming approach). Moreover, with increasing $\theta$, the average cost increases while the standard deviation and the average of worst realizations decrease for each distribution. This is an expected behavior when robust optimization is utilized. 
We note  that Binomial distribution is the least sensitive with respect to $\theta$ as the order quantity (and performance measures) do not change after $\theta \ge 0.05$. On the other hand, Uniform and Poisson distributions are more sensitive  with respect to this parameter. 

We also repeat the experiments with even higher values of $\theta$ and  observe that only the  results corresponding to the Poisson  distribution changes, which we attribute to its right-skewness. However, the order quantities in those cases are very high, which result in overly conservative policies and deteriorated performance measures.

\noindent
\begin{minipage}{\linewidth}
\captionof{table}{Summary results for the Newsvendor Problem with \texttt{Uniform(0, 10)} and $R=100$.}\label{tab:newsUnif}
\centering
\begin{tabular}{ccccc}
\hline
$\theta$ & $y^*$ &   Avg. & St. Dev. &  Worst 10\% \\
\hline
 $0.00$ &        4 &       8.08 &       2.92 &      13.80 \\

 $0.05$ &        4 &       8.08 &       2.92 &      13.80 \\

 $0.10$ &       5 &       8.67 &       2.22 &      12.80 \\

 $0.15$ &       5 &       8.67 &       2.22 &      12.80 \\

 $0.20$ &        5 &       8.67 &       2.22 &      12.80 \\

 $0.25$ &      6	& 9.53	 &1.94	& 12.00 \\

\hline
\end{tabular}
\end{minipage}

\noindent
\begin{minipage}{\linewidth}
\captionof{table}{Summary results for the Newsvendor Problem with \texttt{Binomial(10, 0.5)} and $R=100$.}\label{tab:newsBinom}
\centering
\begin{tabular}{ccccc}
\hline
$\theta$ & $y^*$ &   Avg. & St. Dev. &  Worst 10\% \\
\hline
 $0.00$ &       4 &       6.76 &       2.38 &      11.40 \\

 $0.05$ &        5 &       7.02 &       1.67 &      10.40 \\

 $0.10$ &        5 &       7.02 &       1.67 &      10.40 \\

 $0.15$ &        5 &       7.02 &       1.67 &      10.40 \\

  $0.20$ &        5 &       7.02 &       1.67 &      10.40 \\

 $0.25$ &        5 &       7.02 &       1.67 &      10.40 \\

\hline
\end{tabular}
\end{minipage}   

\noindent
\begin{minipage}{\linewidth}
\captionof{table}{Summary results for the Newsvendor Problem with \texttt{Poisson(5)} and $R=100$.}\label{tab:newsPoi}
\centering
\begin{tabular}{ccccc}
\hline
$\theta$ & $y^*$ &   Avg. & St. Dev. &  Worst 10\% \\
\hline
 $0.00$ &        4 &       7.59 &       3.04 &      13.60 \\

 $0.05$ &         5 &       7.73 &       2.40 &      12.60 \\

 $0.10$ &        5 &       7.73 &       2.40 &      12.60 \\

 $0.15$ &        5 &       7.73 &       2.40 &      12.60 \\

 $0.20$ &     6 &	8.44&	1.86&	12.00 \\

 $0.25$ &     6 &	8.44&	1.86&	12.00 \\

\hline
\end{tabular}
\end{minipage}   

In addition to the summary statistics, we also provide the box plots of the cost realizations in  Figures~\ref{fig:newsUnif}-\ref{fig:newsPoi}  for  Uniform, Binomial and Poisson distributions, respectively. 
We observe that as the robustness level~$\theta$ increases, the median of the cost realizations increases while the range shrinks for each distribution. We also note that the maximum  and upper quartile values decrease for $\theta \in [0.05,0.15]$. This is a desired property since it implies that the risk of the stochastic programming approach ($\theta=0.00$) can be lowered.

\noindent
\begin{minipage}{\linewidth}
\centering
\includegraphics[width=8cm,height=5cm]{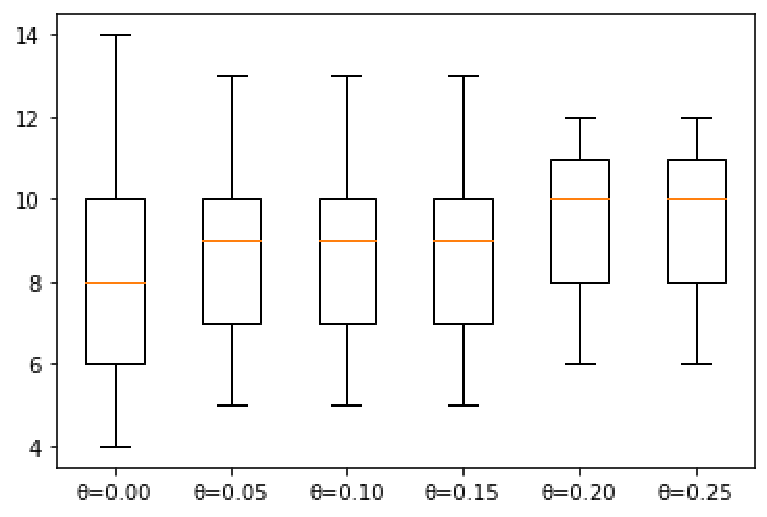}
\captionof{figure}{Box plot of the results for the Newsvendor Problem with \texttt{Uniform(0, 10)} and $R=100$.}\label{fig:newsUnif}
\end{minipage}

\noindent
\begin{minipage}{\linewidth}
\centering
\includegraphics[width=8cm,height=5cm]{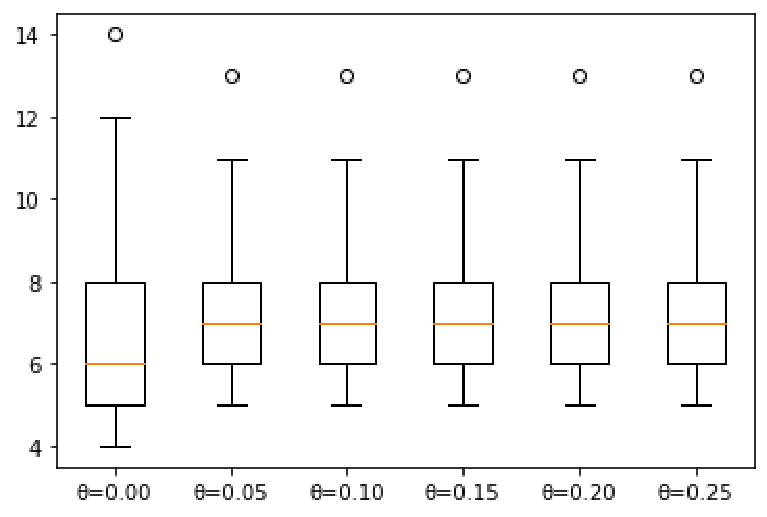}
\captionof{figure}{Box plot of the results for the Newsvendor Problem with \texttt{Binomial(10, 0.5)} and $R=100$.}\label{fig:newsBinom}
\end{minipage}

\noindent
\begin{minipage}{\linewidth}
\centering
\includegraphics[width=8cm,height=5cm]{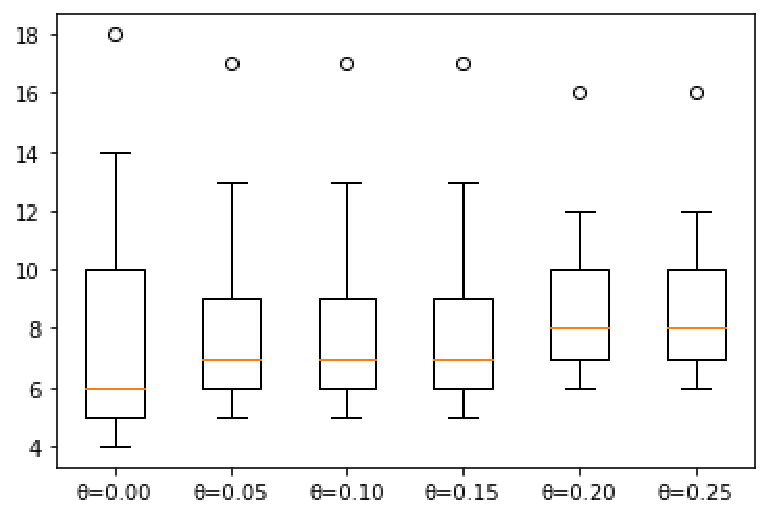}
\captionof{figure}{Box plot of the results for the Newsvendor Problem with \texttt{Poisson(5)} and $R=100$.}\label{fig:newsPoi}
\end{minipage}

\section{Application to the Uncapacitated Facility Location Problem}\label{sec:ufl}

In this section, we analyze  the KL divergence constrained DR version of the Uncapacitated Facility Location (UFL) Problem. 

\subsection{Deterministic Version}
We first remind the reader the deterministic version of the well-known UFL Problem. Suppose that we have $m$ customers each with demand $d^i$, $i\in[m]$. The demand must be satisfied by opening new facilities. There are $n$ potential facilities each with a fixed cost of $f_j$, $j\in[n]$. The unit transportation cost between each customer $i$ and facility $j$ is given as $t_{ij}$, $i\in[m]$, $j\in[n]$. The objective is to minimize the total fixed cost and transportation cost.

The UFL Problem can be modeled as an integer program by defining two sets of binary decision variables. The first set of decision variables, denoted as $y_j$, represent the status of each facility $j$, and the second set of decision variables, denoted as $x_{ij}$, represents the assignment of a customer $i$ to a facility $j$. The complete model is given as follows:
\begin{subequations} \label{eq:ufl}
\begin{align}
\min \ &  \sum_{j=1}^n  \bigg [ f_j y_j +  \sum_{i=1}^m d^i t_{ij} x_{ij} \bigg]  \label{eq:ufl0}\\
\text{s.t.}
\ & \sum_{j=1}^n   x_{ij} = 1    \hspace{1.25cm}  i\in [m]  \label{eq:ufl1} \\
\ & x_{ij}\le y_{j}  \hspace{1.8cm}  i\in [m]; j\in [n]  \label{eq:ufl2} \\
\ & x_{ij} \in\{0,1\}  \hspace{1.25cm}  i\in [m]; j\in [n]  \label{eq:ufl3}\\
\ &  y_{j}\in\{0,1\}  \hspace{1.38cm}  i\in [m]; j\in [n] . \label{eq:ufl4}
\end{align}
\end{subequations}
Here,  constraint~\eqref{eq:ufl1} guarantees that each customer is served by exactly one facility while constraint~\eqref{eq:ufl2} ensures that each customer is served by an open facility.

We point out two useful observations about the UFL Problem. Firstly,  in any feasible solution to problem~\eqref{eq:ufl}, at least one facility must be opened, hence we must have
\begin{equation}\label{eq:valid_y}
\sum_{j=1}^n  y_j  \ge 1.
\end{equation}
Secondly, given any optimal~$y^*$ vector, the optimal objective function value can be computed as 
\begin{equation}\label{eq:opt_ufl_given_y}
 \sum_{j=1}^n  f_j y_j^* +  \sum_{i=1}^m d^i   \min_{j: y_j^* = 1} \{t_{ij} \}  ,
\end{equation}
since each customer can be served by the closest open facility.

\subsection{Distributionally Robust Version}
Now, suppose that we replace the deterministic demand $d^i$ with a random variable $\xi^i$ having an empirical distribution $q^i \in \ri(\Delta^{S_i})$, with realizations $d_s^i$, $s\in[S_i]$. Then, the DR version of the UFL Problem can be modeled as an instance of the generic model~\eqref{eq:generic} with $\mathcal{Y} :=\{ y\in\{0,1\}^n : \eqref{eq:valid_y} \} $ as follows: We choose functions
\[
h(y) =  \sum_{j=1}^n  f_j y_j  
\]
and
\[
H^i(y, \xi^i) = \xi^i \min_{j: y_j =1} \{ t_{ij}\}, \ i=1,\dots,m,
\]
due to~\eqref{eq:opt_ufl_given_y}. 
In the remainder of this subsection, we will obtain the robust counterpart of the KL divergence constrained DR UFL Problem as a dual exponential cone constrained MIP by the help of Lemma~\ref{lem:ufl}.

\subsubsection{A Lemma}

The following lemma will be critical to linearize constraint~\eqref{eq:genericReform2}.

\begin{lemma}\label{lem:ufl}
Let $t \in \mathbb{R}_+^n$ be a given vector and consider the function $g(y) : \{0,1\}^n \to \mathbb{R}$ defined as $g(y) := \min\{ t_j: y_j=1 \}$. Then, for any $y\in\{0,1\}^n$ satisfying~\eqref{eq:valid_y},  we have
\begin{equation}\label{eq:ufl closed}
g(y) =
\max_{l=1,\dots,n} \big\{ t_l - \sum_{j=1}^n y_j \max\{t_l-t_j,0\}  \big\} .
\end{equation}
\end{lemma}
\begin{proof}
Let $y\in\{0,1\}^n$ satisfying~\eqref{eq:valid_y} be given. We  define the following nonempty sets $T:=\{j: y_j =1\} $ and $T^* := \argmin\{t_j : j \in T\}$.
Notice that $g(y) = t_{l}$ for $l \in T^*$.  Also, let us define the quantity
\[
z_l :=  t_l - \sum_{j=1}^n y_j \max\{t_l-t_j,0\}, \ l\in[n],
\]
for convenience.
Notice that we have
\[
\begin{split}
y_j\max\{t_l - & t_j,  0\} =  \\
&
\begin{cases}
0 & \text{ if } j\not\in T\\
0  &  \text{ if } j\in T \text{ and } t_{il} \le t_{ij} \\
 t_{il} - t_{ij}  & \text{ if } j\in T \text{ and } t_{il} > t_{ij}
  \end{cases}.
  \end{split}
\]
This observation helps us to rewrite $z_l$ as
\[
z_l = t_l - \sum_{j\in T :  t_{l} > t_{j}  } (t_l-t_j), \ l\in[n].
\]

Now, we will look at the following cases to compute or bound $z_l$: 

\noindent Case 1: Let  $l^*\in T^*$. Then,  we have $z_{l^*} = t_{l^*}$.

\noindent Case 2: Let  $l\not\in T^*$, and choose any $j^*\in T^*$. Then, we have
\begin{equation*}
\begin{split}
z_l & = t_l - (t_l - t_{j^*}) - \sum_{j\in T\setminus\{j^*\} :  t_{il} > t_{ij}  } (t_l-t_j) \\
& = t_{j^*} - \sum_{j\in T\setminus\{j^*\} :  t_{il} > t_{ij}  } (t_l-t_j)  \\
& \le t_{j^*}.
\end{split}
\end{equation*}
This analysis indicates that 
\[
\max_{l \in [n]} \big\{ t_l - \sum_{j=1}^n y_j \max\{t_l-t_j,0\} \big \} = \max_{l \in [n] } z_l =   t_{l^*},
\]
where $l^* \in T^*$. Hence, we conclude that equation~\eqref{eq:ufl closed} holds true.
\end{proof}
An alternative proof of Lemma~\ref{lem:ufl} can be obtained via LP duality: First, one would write the problem $ \min\{ t_j: y_j=1 \}$ as an IP by introducing additional binary variables $x_j$. Secondly, this IP can be relaxed as an LP due to the totally unimodular structure. Then, the extreme points of the feasible region  of the dual LP can be characterized, enabling the dual LP to be solved in closed form (see dual based arguments in \cite{Erlenkotter1978, Conn1990}).

\subsubsection{The Final Formulation}
Taking into account the special structure of the UFL Problem and utilizing Lemma~\ref{lem:ufl} by setting $g := H^i$ for each $i\in[m]$, we obtain the following dual exponential cone constrained MIP:
 \begin{subequations} \label{eq:uflReform}
\begin{align}
\min \ &  \sum_{j=1}^n  f_j y_j + \sum_{i=1}^m \bigg [ \alpha^i + \epsilon^i \beta^i + \sum_{s=1}^{S_i} q_s^{i} u_s^i \bigg ]  \\
\text{s.t.}
\ & \alpha^{i}  - v_s^i \ge  d_s^{i} \big ( t_{il} - \sum_{j=1}^n  y_j \max\{t_{il}-t_{ij},0\} \big ) \notag \\
\ &  \hspace{2.15cm}  i \in [m]; s \in [S_i]; l \in[n]   \\
\ &   \eqref{eq:genericReform2}-\eqref{eq:genericReform4},  \ \eqref{eq:ufl4}, \ \eqref{eq:valid_y} \notag .
\end{align}
\end{subequations}

\subsection{Computations}

\subsubsection{Experimental Setup}
We utilize Algorithm~\ref{alg:ufl} to compare the effect of robustness level of KL divergence constrained DR version of the UFL Problem. This algorithm is 	quite similar to Algorithm~\ref{alg:news} used for the analysis of the Newsvendor Problem.

\begin{algorithm}
\label{alg:ufl}
\begin{algorithmic}[1]
\REQUIRE A probability distribution $\mathcal{D}$, the number of samples $R$, the set of robustness levels  $\mathcal{T}$.
\STATE Sample $R$  random variates from $\mathcal{D}$ for each customer $i\in[m]$ for training, and obtain the empirical distribution $q^i$ and the maximum KL divergence $\overline\epsilon(q^i)$ for each $i\in[m]$.
\STATE Solve problem~\eqref{eq:uflReform} with  $\epsilon^i := \theta \overline\epsilon(q^i)$ for each $i\in[m]$ and $\theta \in \mathcal{T}$ to obtain a decision vector $y^*(\theta)$. 
\STATE Sample $R$  random variates from $\mathcal{D}$ for each $i\in[m]$ for testing, and then compute the cost realizations for each realization under the decision vector $y^*(\theta)$. 
\end{algorithmic}
\end{algorithm}

\subsubsection{Results}
\label{sec:uflResults}

For this illustration, we assume that there are 12 customers  and three potential facilities located in the unit interval. Their precise locations are given as 
\[
\bigg\{ \frac{2h-1}{36}: h\in[6 ] \bigg \} \cup  \bigg\{\frac{35-2h}{36}: h\in[6] \bigg \},
\]
and  
\[ 
\bigg\{\frac{2H-1}{6}: H\in[3]  \bigg\},
\] 
respectively, and are shown in Figure~\ref{fig:locations}.

\medskip
\noindent
\begin{minipage}{\linewidth}
\centering
\includegraphics[width=8.2cm,height=1.025cm]{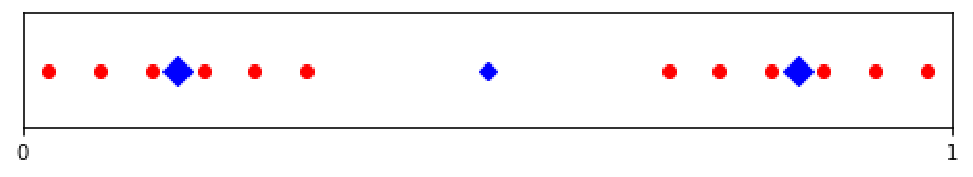}
\captionof{figure}{Locations of the customers (circles) and potential facilities (diamonds).}\label{fig:locations}
\end{minipage}
As evident from the figure, potential facilities are located evenly across the unit interval and there are two clusters of customers which are also distributed evenly in their respective regions.
The fixed cost of opening facilities are given as $f_1=f_3=10$ for the two facilities in the middle of these clusters (marked by a large diamond), and $f_2=5$ for the other facility (marked by a small diamond). Finally, the unit transportation cost between a facility-customer pair is assumed to be equal to the their distance from each other. 

We specify the parameters of Algorithm~\ref{alg:ufl} regarding the generation of random variates similar to that of  Algorithm~\ref{alg:news} as described in Section~\ref{sec:newsResults}.

The summary statistics of our experiments are reported in Tables~\ref{tab:uflUnif}-\ref{tab:uflPoi}  for  Uniform, Binomial and Poisson distributions, respectively. We first observe that the optimal solutions and the performance measures are similar for every distribution, therefore, we will summarize our observations together. 
Due to the choice of parameters and the locations of the facilities and customers as can be seen from Figure~\ref{fig:locations}, there is a fundamental trade-off in this instance: We can i) either open a single facility at the middle of the line segment with the lower fixed cost and serve customers via longer distances,  ii) or open two facilities at the middle of two customer clusters with higher  fixed cost and serve customers via shorter distances. In the stochastic programming approach ($\theta=0.00$), the first  policy becomes optimal whereas in the DR approach ($\theta \ge 0.05$), the second policy becomes optimal. 
We note that considering the ambiguity of the demand distributions increases the average cost only slightly whereas  both the standard deviation and the average of the worst 10\% of the realizations decrease significantly. We remind the reader that the total fixed cost of the stochastic programming approach is only 5 while the fixed cost of the DR approach is 20. This also shows that the corresponding transportation cost, which is affected by the random uncertainty, is significantly smaller in the DR approach.

\noindent
\begin{minipage}{\linewidth}
\captionof{table}{Summary results for the UFL Problem with \texttt{Uniform(0, 10)} and $R=100$.}\label{tab:uflUnif}
\centering
\begin{tabular}{ccccc}
\hline
$\theta$ & $y^*$ &   Avg. & St. Dev. &  Worst 10\% \\
\hline
   $0.00$ & 0,1,0 &      23.11 &       3.47 &      29.19 \\

  $0.05$ & 1,0,1 &      24.52 &       0.92 &      26.10 \\

  $0.10$ & 1,0,1 &      24.52 &       0.92 &      26.10 \\

 $0.15$ & 1,0,1 &      24.52 &       0.92 &      26.10 \\

  $0.20$ & 1,0,1 &      24.52 &       0.92 &      26.10 \\

 $0.25$ & 1,0,1 &      24.52 &       0.92 &      26.10 \\

\hline
\end{tabular}
\end{minipage}

\noindent
\begin{minipage}{\linewidth}
\captionof{table}{Summary results for the UFL Problem with \texttt{Binomial(10, 0.5)} and $R=100$.}\label{tab:uflBinom}
\centering
\begin{tabular}{ccccc}
\hline
$\theta$ & $y^*$ &   Avg. & St. Dev. &  Worst 10\% \\
\hline
     $0.00$ & 0,1,0 &      24.87 &       1.88 &      28.15 \\

     $0.05$ & 1,0,1 &      24.97 &       0.52 &      25.86 \\

    $0.10$ &1,0,1 &      24.97 &       0.52 &      25.86 \\

    $0.15$ & 1,0,1&      24.97 &       0.52 &      25.86 \\

   $0.20$ & 1,0,1&      24.97 &       0.52 &      25.86 \\

  $0.25$ & 1,0,1 &      24.97 &       0.52 &      25.86 \\

\hline
\end{tabular}
\end{minipage}   

\noindent
\begin{minipage}{\linewidth}
\captionof{table}{Summary results for the UFL Problem with \texttt{Poisson(5)} and $R=100$.}\label{tab:uflPoi}
\centering
\begin{tabular}{ccccc}
\hline
$\theta$ & $y^*$ &   Avg. & St. Dev. &  Worst 10\% \\
\hline
       $0.00$ & 0,1,0 &      24.96 &       2.67 &      29.89 \\

       $0.05$ & 1,0,1&      24.99 &       0.74 &      26.34 \\

      $0.10$ & 1,0,1 &      24.99 &       0.74 &      26.34 \\

      $0.15$ & 1,0,1 &      24.99 &       0.74 &      26.34 \\

       $0.20$ &1,0,1&      24.99 &       0.74 &      26.34 \\

       $0.25$ &1,0,1&      24.99 &       0.74 &      26.34 \\

\hline
\end{tabular}
\end{minipage}

In addition to the summary statistics, we also provide the box plots of the cost realizations in  Figures~\ref{fig:uflUnif}-\ref{fig:uflPoi}  for  Uniform, Binomial and Poisson distributions, respectively. 
We observe that  the median of the cost realizations either stays the same or increases slightly in the DR approach while the range shrinks significantly compared to the stochastic programming approach for each distribution. We also note that the maximum  and upper quartile values decrease considerably with the DR approach (especially for Binomial and Poisson distributions).

\noindent
\begin{minipage}{\linewidth}
\centering
\includegraphics[width=8cm,height=5cm]{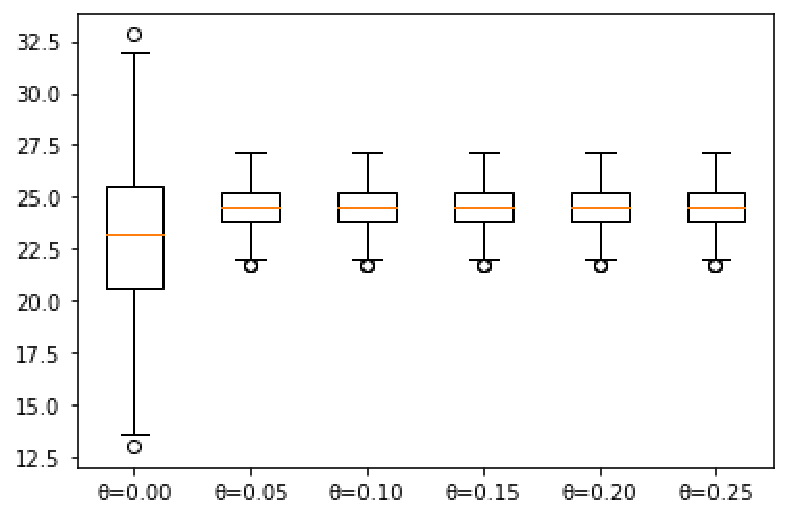}
\captionof{figure}{Box plot of the results for the UFL Problem with \texttt{Uniform(0, 10)} and $R=100$.}\label{fig:uflUnif}
\end{minipage}

\noindent
\begin{minipage}{\linewidth}
\centering
\includegraphics[width=8cm,height=5cm]{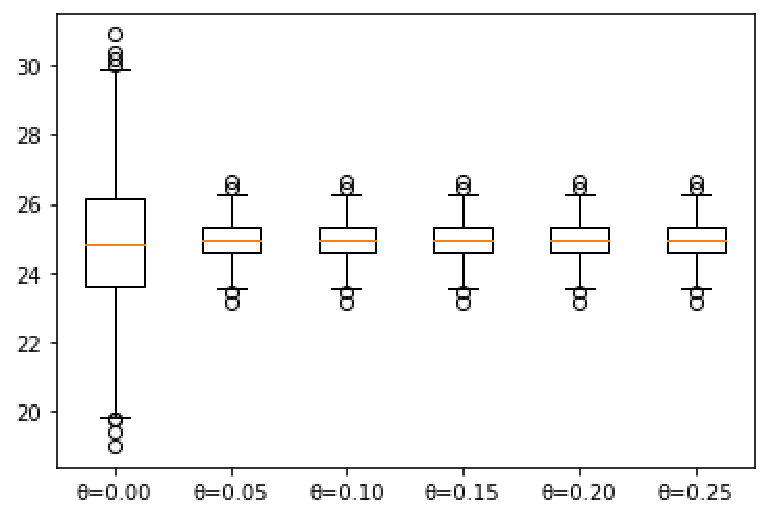}
\captionof{figure}{Box plot of the results for the UFL Problem with \texttt{Binomial(10, 0.5)} and $R=100$.}\label{fig:uflBinom}
\end{minipage}

\noindent
\begin{minipage}{\linewidth}
\centering
\includegraphics[width=8cm,height=5cm]{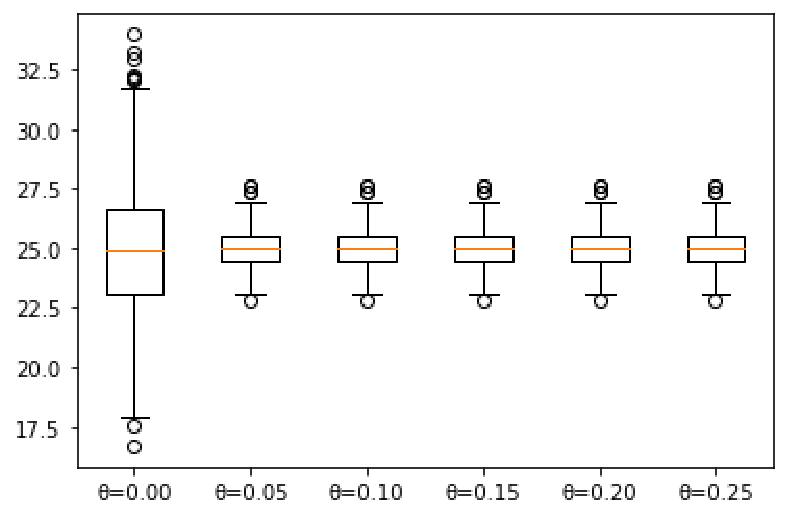}
\captionof{figure}{Box plot of the results for the UFL Problem with \texttt{Poisson(5)} and $R=100$.}\label{fig:uflPoi}
\end{minipage}

\section{Conclusion}
\label{sec:conc}

In this paper, we analyzed the KL divergence constrained DRO problems and proposed their dual exponential cone constrained reformulations utilizing the exponential cone representability property of KL divergence  and Conic Duality. The resulting robust counterpart can be solved by a commercial conic programming solver directly. We specialized our results to the Newsvendor and UFL Problems by providing problem specific reformulations, and conducted a computational analysis comparing the performance of the solutions obtained via DR approach and stochastic programming from different aspects. 
We observed that although the mean and median of the cost realizations deteriorate slightly when the DR approach is preferred; the range, standard deviation and worst case values of the cost realizations improve significantly compared to stochastic programming approach. 

Two future research directions seem promising: We would like to test the success of the proposed method on different problems with real datasets, and adapt our results to the decision-dependent setting, which is a recent active research area in the  DRO literature \cite{Noyan2018, Basciftci2019, Luo2020}.

\section*{Acknowledgments}
\noindent The author would like to thank Dr. Beste Basciftci for her comments on an earlier version of this paper.


\vspace{1.5cm}

\small
\noindent
\textit{\textbf{Burak Kocuk}
is an assistant professor at the Industrial Engineering Program, Sabanc{\i} University. He obtained his BS degrees in Industrial Engineering and Mathematics, and MS degree in Industrial Engineering from Bo\u{g}azi\c{c}i University. He obtained his PhD degree of Operations Research at the School of Industrial and Systems Engineering, Georgia Institute of Technology. Before joining Sabanc{\i} University, he was a postdoctoral fellow at the Tepper School of Business, Carnegie Mellon University. His current research focuses on mixed-integer nonlinear programming  and stochastic optimization problems, from both theoretical and methodological aspects.
}\\
\includegraphics[height=10pt, width=10pt]{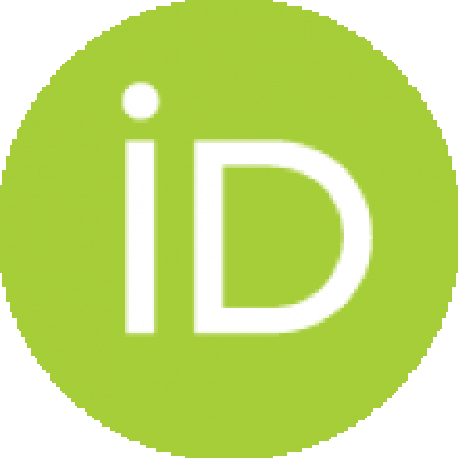} \url{http://orcid.org/0000-0002-4218-1116}

\end{multicols*}
  

\begin{thebibliography}{99}

\normalsize

\bibitem{Basciftci2019}Basciftci, B., Ahmed, S., \& Shen, S. (2019). Distributionally robust facility location problem under decision-dependent stochastic demand.  \textit{Optimization Online}.


\bibitem{BenTal2013}Ben-Tal, A., Den Hertog, D., De Waegenaere, A., Melenberg, B.,  \& Rennen, G. (2013). Robust solutions of optimization problems affected by uncertain probabilities. \textit{Management Science}, 59(2), 341-357.


\bibitem{BN2001}Ben-Tal, A., \& Nemirovski, A. (2001). \textit{Lectures on modern convex optimization: analysis, algorithms, and engineering applications}. Society for Industrial and Applied Mathematics.

\bibitem{BenTal2002}Ben-Tal, A., \& Nemirovski, A. (2002). Robust optimization--methodology and applications. \textit{Mathematical Programming}, 92(3), 453-480.

\bibitem{BEN2009}Ben-Tal, A., El Ghaoui, L., \& Nemirovski, A. (2009). \textit{Robust optimization}. Princeton University Press.

\bibitem{Bertsimas2011}Bertsimas, D., Brown, D. B., \& Caramanis, C. (2011). Theory and applications of robust optimization. \textit{SIAM Review}, 53(3), 464-501.

\bibitem{BL2011}Birge, J. R., \& Louveaux, F. (2011). \textit{Introduction to stochastic programming}. Springer Science \& Business Media.

\bibitem{Chen2018}Chen, Y., Guo, Q., Sun, H., Li, Z., Wu, W., \& Li, Z. (2018). A distributionally robust optimization model for unit commitment based on Kullback–Leibler divergence. \textit{IEEE Transactions on Power Systems}, 33(5), 5147-5160.

\bibitem{Conn1990}Conn, A. R., \& Cornuejols, G. (1990). A projection method for the uncapacitated facility location problem. \textit{Mathematical Programming}, 46(1-3), 273-298.

\bibitem{Dahl2019}Dahl, J., \& Andersen, E. D. (2019). A primal-dual interior-point algorithm for nonsymmetric exponential-cone optimization. \textit{Optimization Online}.

\bibitem{Delage2010}Delage, E. and Ye, Y. (2010). Distributionally robust optimization under moment uncertainty
with application to data-driven problems. \textit{Operations Research}, 58(3):595-612.

\bibitem{Erlenkotter1978}Erlenkotter, D. (1978). A dual-based procedure for uncapacitated facility location. \textit{Operations Research}, 26(6), 992-1009.

\bibitem{Gao2016}Gao, R. \& Kleywegt, A. J. (2016). Distributionally robust stochastic optimization
with Wasserstein distance.  \textit{Optimization Online}.



\bibitem{Hanasusanto2015}Hanasusanto, G. A., Kuhn, D., Wallace, S. W., \& Zymler, S. (2015). Distributionally robust multi-item newsvendor problems with multimodal demand distributions. \textit{Mathematical Programming}, 152(1-2), 1-32.

\bibitem{Hanasusanto2018}Hanasusanto, G. A., \& Kuhn, D. (2018). Conic programming reformulations of two-stage distributionally robust linear programs over Wasserstein balls. \textit{Operations Research}, 66(3), 849-869.

\bibitem{Hu2013}Hu, Z., \& Hong, L. J. (2013). Kullback-Leibler divergence constrained distributionally robust optimization. \textit{Optimization Online}.

\bibitem{Jiang2016}Jiang, R., \& Guan, Y. (2016). Data-driven chance constrained stochastic program. \textit{Mathematical Programming}, 158(1-2), 291-327.

\bibitem{KL1951}Kullback, S., \&   Leibler, R. A.  (1951).  On information and sufficiency. \textit{Annals of Mathematical Statistics}, 22(1), 79-86.

\bibitem{Lam2019}Lam, H. (2019). Recovering best statistical guarantees via the empirical divergence-based distributionally robust optimization. \textit{Operations Research}, 67(4), 1090-1105.

\bibitem{Lee2020}Lee, S., Kim, H., \& Moon, I. (2020). A data-driven distributionally robust newsvendor model with a Wasserstein ambiguity set. \textit{Journal of the Operational Research Society}, 1-19.

\bibitem{Li2018}Li, Z., Wu, W., Zhang, B., \& Tai, X. (2018). Kullback–Leibler divergence-based distributionally robust optimisation model for heat pump day-ahead operational schedule to improve PV integration.  \textit{IET Generation, Transmission \& Distribution}, 12(13), 3136-3144.

\bibitem{Lu2015}Lu, M., Ran, L., \& Shen, Z. J. M. (2015). Reliable facility location design under uncertain correlated disruptions. \textit{Manufacturing \& Service Operations Management}, 17(4), 445-455.

\bibitem{Luo2020}Luo, F., \& Mehrotra, S. (2020). Distributionally robust optimization with decision dependent ambiguity sets. \textit{Optimization Letters}, 1-30.

\bibitem{Mohajerin2018}Mohajerin Esfahani, P. \& Kuhn, D. (2018). Data-driven distributionally robust optimization
using the Wasserstein metric: performance guarantees and tractable reformulations.
\textit{Mathematical Programming}, 171(1):115-166.

\bibitem{MOSEK2020}MOSEK ApS. (2020). MOSEK modeling cookbook.

\bibitem{MOSEK2020py}MOSEK ApS. (2020). MOSEK optimizer API for Python.

\bibitem{Natarajan2018}Natarajan, K., Sim, M., \& Uichanco, J. (2018). Asymmetry and ambiguity in newsvendor models. \textit{Management Science}, 64(7), 3146-3167.

\bibitem{NM1994}Nesterov, Y., \& Nemirovski, A. (1994). \textit{Interior-point polynomial algorithms in convex programming}. Society for Industrial and Applied Mathematics.

\bibitem{Noyan2018}Noyan, N., Rudolf, G., \& Lejeune, M. (2018). Distributionally robust optimization with decision-dependent ambiguity set.  \textit{Optimization Online}.

\bibitem{Popescu2007}Popescu, I. (2007). Robust mean-covariance solutions for stochastic optimization. \textit{Operations Research}, 55(1):98–112.

\bibitem{Santi2018}Santiváñez, J. A., \& Carlo, H. J. (2018). Reliable capacitated facility location problem with service levels. \textit{EURO Journal on Transportation and Logistics}, 7(4), 315-341.

\bibitem{Serrano2015}Serrano, S. A. (2015). Algorithms for unsymmetric cone optimization and an implementation for problems with the exponential cone. PhD Thesis. Stanford University.

\bibitem{SDR2014}Shapiro, A., Dentcheva, D., \& Ruszczyński, A. (2014). \textit{Lectures on stochastic programming: modeling and theory}. Society for Industrial and Applied Mathematics.

\bibitem{Wiesemann2014}Wiesemann, W., Kuhn, D., \& Sim, M. (2014). Distributionally robust convex optimization. \textit{Operations Research}, 62(6), 1358-1376.

\bibitem{Xie2018}Xie, W., \& Ahmed, S. (2018). On deterministic reformulations of distributionally robust joint chance constrained optimization problems. \textit{SIAM Journal on Optimization}, 28(2), 1151-1182.

\bibitem{Xie2019}Xie, W. (2019). On distributionally robust chance constrained programs with Wasserstein distance. \textit{Mathematical Programming}, 1-41.

\bibitem{Zymler2013}Zymler, S., Kuhn, D., \& Rustem, B. (2013). Distributionally robust joint chance constraints with second-order moment information. \textit{Mathematical Programming}, 137(1-2), 167-198.











\end{thebibliography}
\end{document}